\documentclass[lettersize,journal]{IEEEtran}
\usepackage{amsmath,amsfonts}
\usepackage{algorithmic}
\usepackage{algorithm}
\usepackage{array}
\usepackage[caption=false,font=normalsize,labelfont=sf,textfont=sf]{subfig}
\usepackage{textcomp}
\usepackage{stfloats}
\usepackage{url}
\usepackage{verbatim}
\usepackage{graphicx}
\usepackage{cite}
\usepackage{hyperref}
\usepackage{amsthm}
\usepackage{fancyhdr}
\usepackage{enumerate}
\usepackage{tabularx}

 \usepackage{booktabs}
 \usepackage{balance}
\usepackage{float} 

\newtheorem{theorem}{Theorem}
\newtheorem{proposition}{Proposition}

\newcounter{remark}
\newenvironment{remark}
{\refstepcounter{remark}\par\medskip\noindent\textbf{Remark \theremark.}\ }
{\par\medskip}

\newcommand{\secref}[1]{Section~\ref{#1}}
\newcommand{\figref}[1]{Fig.~\ref{#1}}

\newcommand{\thrmref}[1]{Theorem~\ref{#1}}

\newcommand{\propref}[1]{Proposition~\ref{#1}}

\begin{document}
\title{Regret-weighted Bayes Fusion for Distributed Experimental Design}
\author{Nagananda K G, Lav R. Varshney,~\IEEEmembership{Senior Member,~IEEE,} Pramod K. Varshney,~\IEEEmembership{Life Fellow,~IEEE}  
\thanks{Nagananda K G is with Fariborz Maseeh Department of Mathematics and Statistics, Portland State University, Portland, OR 97201, USA. (email: \texttt{nanda@pdx.edu}).}
\thanks{Lav R. Varshney is with the AI Innovation Institute, Stony Brook University, Stony Brook, NY 11794, USA, and with Brookhaven National Laboratory, Upton, NY 11973, USA. (email: \texttt{lav.varshney@stonybrook.edu}).}
\thanks{Pramod K. Varshney is with the Department of Electrical Engineering and Computer Science, Syracuse University, Syracuse, NY 13244 USA. (email: \texttt{varshney@syr.edu}).}
}



\maketitle

\begin{abstract}
We study distributed experimental design with multiple candidate experiments, where local sites possess only partial information and transmit design recommendations to a fusion center. Unlike centralized design, in which the experiment that maximizes expected information gain can be selected directly, distributed design requires combining heterogeneous and potentially conflicting local recommendations. Formulating as a multi-class Bayes fusion problem, centralized oracle design is treated as an unknown label and each site is characterized by a local recommendation mechanism.  The proposed fusion rule minimizes posterior expected information regret, rather than merely maximizing the number of local votes or the posterior probability (MAP) of the oracle label.  This distinction is essential because different incorrect experimental choices may incur different losses in information gain.  We show that majority vote is optimal only under restrictive symmetry assumptions and can otherwise be strictly suboptimal. Regret-weighted multi-class Chernoff bounds are derived to identify the pairwise separations governing distributed design performance.  Numerical studies identify two operational regimes: MAP is effective when oracle-label accuracy and information regret are aligned, while regret-weighted Bayes fusion reduces information loss when the most probable oracle label is not the lowest-regret decision. 
\end{abstract}

\section{Introduction}
Distributed experimental design arises when the information needed to select an experiment is dispersed across several sites, agents, sensors, laboratories, clinics, or data repositories \cite{Hamburg2021,Alexander2023,McMahan2017,Tamang2026,Roy2026}. In a centralized setting, one evaluates the available candidate experiments using a global utility criterion, such as expected information gain (EIG) \cite{Chaloner1995}, and selects the experiment that maximizes the criterion. In a distributed setting, however, no single site has access to the full information needed to evaluate the global utility.  Each site observes only its own local data structure, population characteristics, measurement constraints, cost profile, or scientific objective; see, e.g., \cite{Meinert1986,Jong2022,Li2023,Slautin2024,Lobel2025,Slautin2025,Kijewski2026}.  Consequently, each site may form a local experiment recommendation that is informative, but incomplete from a global perspective.  We focus on the case in which the fusion center must select a common experiment to be implemented across sites, as occurs when coordinated resource allocation preclude site-specific experimental choices; see, e.g., \cite{Burman2001}.  The central question is then how to combine these local recommendations into a single global experimental choice that approximates the decision that would have been made under centralized access to all information. 

A simple approach is to treat the problem as one of voting \cite{May1952,Fishburn1973,Black1987,AustenSmith1996}.  Each site recommends one experiment and the fusion center selects the experiment receiving the most recommendations. Although this approach is easy to implement and appealing in symmetric settings, it ignores several important features of distributed experimental design. Sites may have different levels of reliability. Some sites may be highly informative about the globally optimal design, while others may be only weakly informative. Moreover, different errors may have different consequences. Selecting an experiment that is nearly as informative as the global optimum should not be penalized in the same way as selecting an experiment that is substantially less informative. Thus, the distributed experimental design problem is not merely an aggregation problem; it is a decision problem under uncertainty, where local recommendations must be interpreted in light of site-specific reliability and the scientific cost of choosing a suboptimal experiment.

In our earlier binary formulation, the distributed design problem involved two candidate experiments \cite{Nagananda2026c}. Each local site compared the two candidates and sent a binary recommendation to the fusion center. The fusion problem then reduced to deciding which of the two experiments was more likely to coincide with the centralized oracle choice, with EIG as the utility.  This led to a Bayes-optimal fusion rule with a threshold-based likelihood-ratio structure.  Each local recommendation contributed a piece of evidence in favor of one experiment or the other, and the fusion center compared the accumulated evidence against a loss-adjusted decision threshold.  The binary setup \cite{Nagananda2026c} showed that majority voting is generally not optimal: the optimal decision must account for prior plausibility, local reliability, and asymmetric losses.   

The present paper extends the binary case to the setting in which there are multiple candidate experiments.  When there are only two experiments, the fusion problem can be expressed as a threshold comparison between two posterior risks.  With multiple candidate experiments, the problem becomes a multi-class Bayes decision problem \cite{DeGroot2004,Bishop2006,Hastie2009}.  The fusion center must compare several possible global designs simultaneously, and the consequences of errors are no longer summarized by a single pair of error costs \cite{Kittler1998,Ruta2000,Kittler2003}.  Instead, every possible selected experiment must be compared against every possible oracle experiment. This creates a full loss structure over pairs of candidate designs and requires a multi-class fusion rule.

The main conceptual contribution of this paper is to formulate distributed experimental design with multiple candidate experiments as a regret-weighted Bayes fusion problem. The centralized oracle is defined as the experiment that would have been selected if all information were centrally available. Each local site produces a recommendation based on its own information.  After receiving the local recommendations, the fusion center evaluates how likely each centralized oracle choice is, given the local recommendations, and how much information would be lost by selecting each candidate experiment. The resulting rule selects the experiment with minimum posterior expected information regret.  This criterion directly reflects the purpose of experimental design: the goal is not merely to identify the oracle label correctly, but to preserve as much centralized information gain as possible.  

This regret-based viewpoint is important because of the following reason.  In ordinary classification, all incorrect labels are often treated equally, at least under the standard zero-one loss.  In experimental design, however, the candidate experiments may be ordered or structured by their information values.  Two experiments may have very similar EIGs, while a third may be substantially inferior.  A fusion rule that treats all errors equally may, therefore, make decisions that are statistically defensible as classifications but scientifically inefficient as designs.  By using information regret as the loss, the proposed framework distinguishes between harmless and consequential errors.  For example, this is especially relevant in biomedical applications and multi-institutional settings, where experimental choices can have substantial downstream effects on cost, sample size, measurement burden, and scientific precision \cite{Meinert2012,Ram2012,Piantadosi2017}.

The multiple-experiment setting also clarifies why majority vote is generally inadequate. Majority vote implicitly assumes that all sites have equal reliability, all local errors have the same structure, all candidate experiments are equally plausible a priori, and all wrong choices are equally costly. These assumptions are rarely appropriate in distributed experimental design. Some sites may have data distributions that are more representative of the centralized objective. Some may be better able to distinguish particular pairs of candidate experiments. Others may be informative for one design comparison but nearly uninformative for another. The proposed Bayes fusion rule accommodates these differences through site-specific recommendation behavior and a regret-based loss structure. As a result, a recommendation from a highly informative site may rationally outweigh several recommendations from weakly informative sites.

Beyond the binary case \cite{Nagananda2026c}, the multiple-experiment setup reveals that fusion performance is governed by pairwise separations among candidate oracle labels.  With two candidate experiments, there is only one fundamental comparison: one must distinguish the first experiment from the second. With many candidate experiments, the fusion problem contains many pairwise comparisons, and these comparisons need not be equally difficult or equally important. Two experiments may be difficult to distinguish but nearly equivalent in information value; another pair may be easy to distinguish but associated with large regret if confused; yet another pair may be both difficult to distinguish and scientifically consequential.  The performance of a distributed design rule is therefore governed by regret-weighted pairwise separability among candidate experiments, rather than by a single aggregate measure of how often local sites identify the oracle design correctly.

The following are the main contributions of this paper.
\begin{enumerate}[(1)]
\item We introduce the finite multiple-experiment fusion model, define the centralized oracle design, and characterize local recommendation mechanisms.  In \thrmref{thrm:Bayes_optimal}, we derive the Bayes-optimal fusion rule under a general design-selection loss.  In \thrmref{thrm:finite_L}, we specialize the result to information-regret loss and show how the binary fusion rule arises as a special case. This comparison demonstrates both continuity and departure: the binary rule is recovered exactly when the candidate set contains two experiments, but the general case requires comparing multiple posterior risks rather than evaluating a single threshold.  In \propref{prop:majority_vote_optimal}, we identify conditions under which majority vote coincides with Bayes fusion and provide explicit examples showing that majority vote can be strictly suboptimal.

\item In \thrmref{thrm:multiclass_Chernoff_bound} and \thrmref{thrm:CB_info_regret_fusion}, we derive performance bounds for the proposed fusion procedure.  In the binary case, performance can be controlled through a Chernoff-type coefficient measuring how well local recommendations distinguish the two possible oracle choices. In the multiple-experiment case, this becomes a family of pairwise coefficients, one for each pair of candidate oracle labels. The resulting bound shows that the regret risk decreases when the distributed sites collectively separate the most consequential pairs of candidate experiments. This provides an interpretable design principle: distributed systems should not merely increase the number of local recommendations; they should improve the reliability of the local recommendation for those design comparisons where an error would cause substantial loss of information.
\end{enumerate}

We complement the theoretical development using two distributed experimental-design settings. The first is a Gaussian linear design model, where each local site has a local precision profile over the candidate experiments and the centralized oracle is determined by the aggregate information gain across all sites. The second is a binary-response design model, where each candidate experiment corresponds to a response design point and the centralized utility depends on the joint predictive distribution of the binary responses. In both settings, local recommendations can conflict with the centralized optimum, and na\"{\i}ve aggregation can lead to inefficient design choices.  We compare majority voting, maximum a posteriori (MAP) fusion, and the proposed regret-weighted Bayes fusion rule in both settings.  Majority vote selects the candidate experiment that receives the largest number of local recommendations.  MAP fusion selects the candidate that is most likely to coincide with the centralized oracle design.  Regret-Bayes fusion selects the candidate that minimizes posterior expected information loss.  

The performance of each rule is evaluated using two complementary criteria:  (i)  Oracle-selection accuracy, that quantifies the proportion of test instances in which the selected experiment coincides exactly with the centralized oracle.  This is the natural performance measure for MAP fusion.  (ii) Average information regret, that measures the average loss in centralized EIG caused by selecting the fused experiment instead of the oracle experiment.  This should be the primary criterion for experimental design, because the objective is not classification accuracy by itself, but preservation of information.  This distinction is essential because a rule that frequently identifies the oracle label need not minimize the information loss caused by incorrect experimental choices.  

The numerical studies consider two settings: baseline and calibrated loss-sensitive.  In the baseline setup, oracle-label accuracy and information regret are closely aligned.  In the calibrated loss-sensitive setup, some candidate experiments are difficult to distinguish from local recommendations, while certain wrong selections incur substantially larger information regret.  The results, therefore, reveal two distinct regimes and provide guidance on which fusion rule should be used:
\begin{enumerate}[(a)]
\item MAP fusion is preferable when the primary goal is to recover the centralized oracle label, or when different wrong experimental choices have roughly comparable information losses.  In this regime, oracle-label accuracy and information regret are closely aligned, so selecting the most probable oracle experiment is a natural objective.  
\item Regret-weighted Bayes fusion is preferable when the objective is to preserve centralized information gain and when different wrong choices have unequal consequences.  In this regime, the most probable oracle label need not minimize posterior expected information loss, and a rule that accounts for the information-regret matrix can yield better design performance. 
\end{enumerate}
Thus, MAP is appropriate for oracle recovery, whereas regret-weighted Bayes fusion is appropriate for loss-sensitive distributed experimental design. 


\section{System model}\label{sec:system_model}
Let $\Xi = \{\xi_1, \ldots, \xi_L\}$ be a finite collection of $L$ candidate experiments. The unknown parameter is denoted by $\theta$, with prior density or probability mass function $p(\theta)$. There are $M$ distributed agents, with agent $m$ having local statistical model $p_m(y_m \mid \theta, \xi_\ell)$, $m=1, \ldots, M$, $\ell=1, \ldots, L$.  If experiment $\xi_\ell$ were implemented centrally, the joint model would be
\begin{equation*}
p(y_1, \ldots, y_M \mid \theta, \xi_\ell) = \prod_{m=1}^M p_m(y_m \mid \theta, \xi_\ell),
\end{equation*}
assuming conditional independence of the observations given $(\theta, \xi_\ell)$.  For the candidate experiment $\xi_\ell$, the centralized EIG is defined as $U(\xi_\ell) = \mathbb{E}_{Y_{1:M} \mid \xi_\ell} \left[ D_{\mathrm{KL}} \{ p(\theta \mid Y_{1:M}, \xi_\ell) \mid p(\theta) \} \right] = I(\theta; Y_{1:M} \mid \xi_\ell)$, the mutual information between the unknown parameter and the full centralized data under experiment $\xi_\ell$.  The centralized oracle design is
\begin{equation}
B = \arg\max_{1 \le \ell \le L} U(\xi_\ell).
\label{eq:central_oracle}
\end{equation}
When ties occur, we use a fixed deterministic tie-breaking rule; for example, the smallest maximizing index
\begin{equation*}
B = \min \left\{ \ell : U(\xi_\ell) = \max_{1 \le r \le L} U(\xi_r) \right\}.
\label{eq:central_oracle_tie}
\end{equation*}

Agent $m$ computes the local EIG $U_m(\xi_\ell) = \mathbb{E}_{Y_m \mid \xi_\ell} \left[ D_{\mathrm{KL}} \{ p_m(\theta \mid Y_m, \xi_\ell) \mid p(\theta) \} \right]$, and reports the local design recommendation 
\begin{equation*}
D_m = \arg\max_{1 \le \ell \le L} U_m(\xi_\ell).
\end{equation*}
Again, ties are resolved by a fixed deterministic rule. Thus, $D_m \in \{1, \ldots, L\}$, $m = 1, \ldots, M$.  The fusion center observes only the local design messages $D = (D_1, \ldots, D_M) \in \{1, \ldots, L\}^M$ and must choose one global experiment $\delta(D) \in \{1, \ldots, L\}$.  The goal is to construct a fusion rule $\delta$ whose selected experiment has centralized information gain as close to that of the centralized oracle as possible.

In experimental design, different suboptimal choices can have different scientific consequences.  For instance, choosing the second-best experiment may be nearly harmless, whereas choosing a much weaker experiment may incur substantial loss.  A natural loss function in such scenarios is the information-gain regret.  We define the information-gain regret of selecting $\xi_i$ when the centralized oracle is $\xi_j$ by $\Delta_{ij} = U(\xi_j) - U(\xi_i)$, whenever $j = B$.  Since $j = B$ implies $\xi_j$ is oracle-optimal, we have $U(\xi_j) \ge U(\xi_i)$, and hence $\Delta_{ij} \ge 0$.  Also, $\Delta_{jj} = 0$.  More generally, if the centralized utilities are treated as random because they depend on an unobserved population state, data-generating environment, or collection of site-specific conditions, we define $\Delta_{ij} = \mathbb{E} \left[ U(\xi_j) - U(\xi_i) \mid B = j \right]$, where $B$ is given by \eqref{eq:central_oracle}. This gives a deterministic regret matrix $\Delta = (\Delta_{ij})_{1 \le i, j \le L}$, with $\Delta_{jj} = 0$, $\Delta_{ij} \ge 0$.  Thus, the fusion center should not merely estimate $B$; rather, it should choose the design that minimizes posterior expected information regret.

For a finite $L$, we can reduce the problem to the oracle label $B \in \{1, \ldots, L\}$.  Since the candidate set is finite, the centralized design problem can be summarized by the index of the experiment that would be chosen under full information. Thus, instead of carrying the entire vector of centralized EIGs through the fusion rule, we encode the oracle decision by the label $B \in \{1, \ldots, L\}$, where $B = j$ means $\xi_j$ is the experiment selected by the centralized oracle.  Let $\rho_j = \mathbb{P}(B = j)$, $j = 1, \ldots, L$, where $\rho_j \ge 0$ and $\sum_{j = 1}^L \rho_j = 1$ denote the prior probability that $\xi_j$ is the experiment selected by the oracle.  For agent $m$, define the local reliability matrix or the recommendation channel $Q_m(a \mid j) = \mathbb{P}(D_m = a \mid B = j)$, $a, j \in \{1, \ldots, L\}$.  Thus, $Q_m(\cdot \mid j)$ is the conditional distribution of agent $m$'s recommendation when the centralized oracle design is $\xi_j$. For each fixed $m$ and $j$, $Q_m(a \mid j) \ge 0$ and $\sum_{a=1}^L Q_m(a \mid j) = 1$.  We assume conditional independence of local design messages given the oracle label:
\begin{equation*}
\mathbb{P}(D_1 = d_1, \ldots, D_M = d_M \mid B = j) = \prod_{m=1}^M Q_m(d_m \mid j).
\end{equation*}
This assumption signifies that once the centralized oracle label $B$ is fixed, the remaining randomness in the local design recommendations is agent-specific.  For a general design-selection loss $C(i, j)$ incurred by selecting $\xi_i$ when the centralized oracle design is $\xi_j$, we assume $C(i, j) \ge 0$, $C(j, j) = 0$.

\section{Main results}\label{sec:main_results}
We now state the Bayes rule in terms of conditional EIG.  Let $G_\ell = U(\xi_\ell)$, $\ell=1, \ldots, L$, denote the centralized EIG for experiment $\xi_\ell$, and $G = (G_1, \ldots, G_L)$.  The oracle EIG value is $G_B = \max_{1 \le \ell \le L} G_{\ell}$.  If the fusion center selects $\xi_{\ell}$, the information-gain regret is $G_B - G_{\ell}$.  A fusion rule is a measurable map $\delta : \{1, \ldots, L\}^M \to \{1, \ldots, L\}$.  Its Bayes risk under information-gain regret is $R(\delta) = \mathbb{E} \left[ G_B - G_{\delta(D)} \right]$.
\begin{theorem}
The Bayes-optimal fusion rule maximizes posterior expected centralized information gain, and is given by 
\begin{equation}
\delta^*(d) = \arg\max_{1 \le i \le L} \mathbb{E} \left[ G_i \mid D = d \right].
\label{eq:Bayes_optimal}
\end{equation}
\label{thrm:Bayes_optimal}
\end{theorem}
\vspace{-0.35cm}
\begin{IEEEproof} 
For a fixed observed message vector $D = d$, the conditional risk of selecting design $\xi_i$ is $r(i \mid d) = \mathbb{E} \left[ G_B - G_i \mid D = d \right]$.  Since $G_B$ does not depend on the selected action $i$, we have $r(i \mid d) = \mathbb{E} \left[ G_B \mid D = d \right] - \mathbb{E} \left[ G_i \mid D = d \right]$,  where the first term is common to all $i$. Therefore, minimizing $r(i \mid d)$ over $i$ is equivalent to maximizing $\mathbb{E} \left[ G_i \mid D = d \right]$.  Hence any rule satisfying \eqref{eq:Bayes_optimal} minimizes the conditional risk for every $d$, and therefore, minimizes the unconditional Bayes risk $R(\delta)$.  
\end{IEEEproof} 
\thrmref{thrm:Bayes_optimal} says that the fusion center should choose the experiment with the largest posterior expected centralized information gain, not necessarily the experiment receiving the largest number of local votes.  In the following, we show that this result can be specialized to the information-regret loss. 

\begin{theorem}
For a loss matrix $C(i, j)$, the Bayes-optimal fusion rule in the distributed formulation is
\begin{equation*}
\delta_C^*(d_1, \ldots, d_M) = \arg\min_{1 \le i \le L} \sum_{j=1}^L C(i, j) \rho_j \prod_{m=1}^M Q_m(d_m \mid j).
\end{equation*}
\label{thrm:finite_L}
\end{theorem}
\begin{IEEEproof} 
Using Bayes' rule, we have
\begin{equation*}
\begin{split}
\mathbb{P}(B = j \mid D = d) &= \frac{\rho_j \mathbb{P}(D = d \mid B = j)}{\sum_{k=1}^L \rho_k \mathbb{P}(D = d \mid B = k)} \\
&\stackrel{(a)}= \frac{\rho_j \prod_{m=1}^M Q_m(d_m \mid j)}{\sum_{k=1}^L \rho_k \prod_{m=1}^M Q_m(d_m \mid k)},
\end{split}
\end{equation*}
where $(a)$ follows from conditional independence.  The conditional risk of selecting $\xi_i$ after observing $D = d$ is
\begin{equation*}
\begin{split}
r_C(i \mid d) &= \mathbb{E} \left[ C(i, B) \mid D = d \right] \\
&= \sum_{j=1}^L C(i, j) \mathbb{P}(B = j \mid D = d) \\
&= \frac{\sum_{j=1}^L C(i, j) \rho_j \prod_{m=1}^M Q_m(d_m \mid j)}{\sum_{k=1}^L \rho_k \prod_{m=1}^M Q_m(d_m \mid k)},
\end{split}
\end{equation*}
where the denominator does not depend on $i$. Hence, minimizing $r_C(i \mid d)$ over $i$ is equivalent to minimizing
\begin{equation*}
\sum_{j=1}^L C(i, j) \rho_j \prod_{m=1}^M Q_m(d_m \mid j).
\end{equation*}
Therefore, we have 
\begin{equation*}
\delta_C^*(d) = \arg\min_{1 \le i \le L} \sum_{j=1}^L C(i, j) \rho_j \prod_{m=1}^M Q_m(d_m \mid j).
\end{equation*}
\end{IEEEproof} 
Note that if all wrong design choices are treated equally, then $C(i, j) = \mathbf{1}_{i \neq j}$. The conditional risk is
\begin{equation*}
r_C(i \mid d) = \mathbb{P}(B \neq i \mid D = d) = 1 - \mathbb{P}(B = i \mid D = d).
\end{equation*}
Thus, minimizing conditional risk is equivalent to maximizing the posterior probability of the oracle label.  Therefore,
\begin{equation*}
\delta_{\mathrm{MAP}}(d) = \arg\max_{1 \le j \le L} \mathbb{P}(B = j \mid D = d).
\end{equation*}
Under ordinary $0-1$ oracle-label loss (the finite-class fusion), we can write  
\begin{align*}
\delta_{\mathrm{MAP}}(d) &= \arg\max_{1 \le j \le L} \rho_j \prod_{m=1}^M Q_m(d_m \mid j) \\
&\stackrel{(b)}= \arg\max_{1 \le j \le L} \left\{ \log \rho_j + \sum_{m=1}^M \log Q_m(d_m \mid j) \right\},
\end{align*}
where $(b)$ holds if all entries $Q_m(d_m \mid j) > 0$.  This is the $L$-candidate analog of the binary fusion rule. For $L > 2$, there is no single scalar threshold\textemdash the fusion center compares $L$ posterior scores.

For experimental design, the more appropriate loss is the information-regret matrix $\Delta_{ij} = \mathbb{E} \left[ U(\xi_j) - U(\xi_i) \mid B = j \right]$, where $\Delta_{ij}$ is the expected centralized information loss incurred by choosing $\xi_i$ when $\xi_j$ is the oracle experiment.  Letting $C(i, j) = \Delta_{ij}$ in \thrmref{thrm:finite_L} gives the Bayes-optimal information-regret fusion rule:
\begin{equation*}
\delta_{\Delta}^*(d_1, \ldots, d_M) = \arg\min_{1 \le i \le L} \sum_{j=1}^L \Delta_{ij} \rho_j \prod_{m=1}^M Q_m(d_m \mid j).
\end{equation*}
In posterior form, we can write 
\begin{equation*}
\delta_{\Delta}^*(d) = \arg\min_{1 \le i \le L} \sum_{j=1}^L \Delta_{ij} \mathbb{P}(B = j \mid D = d).
\end{equation*}
This rule differs from ordinary MAP fusion.  It may intentionally choose a design whose posterior probability of being oracle-optimal is not maximal, if that design has smaller posterior expected information regret.  This is important in distributed experimental design: the fusion center should not simply seek the experiment that is most likely to be best in terms of label accuracy.  It should instead seek the experiment that minimizes the expected loss of centralized information.

\subsection{The binary setting, $L = 2$}
For the binary case ($L = 2$), we write the two candidate experiments as $\xi_0$ and $\xi_1$.  With $B \in \{0, 1\}$, let $\rho_1 = \mathbb{P}(B = 1)$, $\rho_0 = \mathbb{P}(B = 0)$. For site $m$, define $p_m = \mathbb{P}(D_m = 1 \mid B = 1)$, $q_m = \mathbb{P}(D_m = 1 \mid B = 0)$.  Then, $\mathbb{P}(D_m = 0 \mid B = 1) = 1 - p_m$, and $\mathbb{P}(D_m = 0 \mid B = 0) = 1 - q_m$.  For an observed vector $d = (d_1, \ldots, d_M) \in \{0, 1\}^M$, the posterior odds are
\begin{equation*}
\frac{\mathbb{P}(B = 1 \mid D = d)}{\mathbb{P}(B = 0 \mid D = d)} = \frac{\rho_1}{\rho_0} \prod_{m=1}^M \frac{\mathbb{P}(D_m = d_m \mid B = 1)}{\mathbb{P}(D_m = d_m \mid B = 0)}.
\end{equation*}
Therefore,
\begin{align*}
& \log \frac{\mathbb{P}(B = 1 \mid D = d)}{\mathbb{P}(B = 0 \mid D = d)}\\ 
&= \log \frac{\rho_1}{\rho_0} + \sum_{m=1}^M \left[ d_m \log \frac{p_m}{q_m} + (1 - d_m) \log \frac{1 - p_m}{1 - q_m} \right].
\end{align*}

Suppose the loss of choosing $\xi_1$ when $\xi_0$ is oracle-optimal is $C(1, 0)$, and the loss of choosing $\xi_0$ when $\xi_1$ is oracle-optimal is $C(0, 1)$. The conditional risk of choosing $\xi_1$ is $C(1, 0) \mathbb{P}(B = 0 \mid D = d)$, whereas the conditional risk of choosing $\xi_0$ is $C(0, 1) \mathbb{P}(B = 1 \mid D = d)$. Hence the Bayes rule chooses $\xi_1$ whenever $C(1, 0) \mathbb{P}(B = 0 \mid D = d) < C(0, 1) \mathbb{P}(B = 1 \mid D = d)$, which yields
\begin{equation*}
\frac{\mathbb{P}(B = 1 \mid D = d)}{\mathbb{P}(B = 0 \mid D = d)} > \frac{C(1, 0)}{C(0, 1)}.
\end{equation*}
Taking logarithms, the rule chooses $\xi_1$ whenever
\begin{align*}
& \log \frac{\rho_1}{\rho_0} + \sum_{m=1}^M \left[ d_m \log \frac{p_m}{q_m} + (1 - d_m) \log \frac{1 - p_m}{1 - q_m} \right] \\
&> \log \frac{C(1, 0)}{C(0, 1)}.
\end{align*}
Thus, the binary threshold rule is a special case of the finite-$L$ Bayes fusion rule \cite{Nagananda2026c}.

\subsection{Suboptimality of majority vote}
The majority-vote rule is
\begin{equation*}
\delta_{\mathrm{MV}}(d) = \arg\max_{1 \le a \le L} \sum_{m=1}^M \mathbf{1}_{d_m = a}.
\end{equation*}
It chooses the experiment receiving the largest number of local recommendations.  Majority vote ignores three quantities that are central to Bayes fusion: the prior probabilities $\rho_j$, the site-specific reliability matrices $Q_m(a \mid j)$, and the regret matrix $\Delta_{ij}$. It is optimal only under restrictive symmetry assumptions.  In the following, we show that when the reliability matrix is symmetric, the majority vote is Bayes-optimal. 
\begin{proposition}
Suppose we consider the $0-1$ loss.  Let $\rho_1 = \cdots = \rho_L = 1/L$, and suppose that every site has the same symmetric reliability matrix
\begin{equation*}
Q_m(a \mid j) = \begin{cases} \alpha, & a=j, \\ \beta, & a \neq j, \end{cases}
\end{equation*}
where $\beta = (1 - \alpha)/(L - 1)$ and $\alpha > \beta$.  Then, the Bayes rule is the majority voting rule. 
\label{prop:majority_vote_optimal}
\end{proposition}
\begin{IEEEproof} 
For an observed message vector $d = (d_1, \ldots, d_M)$, define $N_j(d) = \sum_{m=1}^M \mathbf{1}_{d_m = j}$. This is the number of sites recommending $\xi_j$.  Under the stated assumptions, $\mathbb{P}(D = d \mid B = j) = \prod_{m=1}^M Q_m(d_m \mid j)$. If $d_m = j$, the contribution is $\alpha$. If $d_m \neq j$, the contribution is $\beta$. Therefore, $\mathbb{P}(D = d \mid B = j) = \alpha^{N_j(d)} \beta^{M - N_j(d)}$.  Since the priors are uniform, $\mathbb{P}(B = j \mid D = d) \propto \alpha^{N_j(d)} \beta^{M - N_j(d)}$. Taking logarithms,
\begin{align*}
\log \mathbb{P}(B = j \mid D = d) &= \text{constant} + N_j(d) \log \alpha\\ 
&+ (M - N_j(d)) \log \beta \\
&= \text{constant} + N_j(d) \log \frac{\alpha}{\beta}.
\end{align*}
Since $\alpha > \beta$, the posterior probability is maximized by maximizing $N_j(d)$.  Hence, the Bayes rule chooses the experiment with the largest number of votes, which is majority voting. 
\end{IEEEproof} 
\propref{prop:majority_vote_optimal} explains why majority vote can appear natural, but it also shows its limitations. If priors are nonuniform, sites have different reliabilities, errors are asymmetric, or losses are not $0-1$, majority vote is generally not Bayes-optimal.  This is illustrated via the following example. 

Let $L = 3$ and $M = 3$. Suppose the prior over oracle labels is uniform: $\rho_1 = \rho_2 = \rho_3 = \frac{1}{3}$. Assume $0-1$ loss. Let sites 1 and 2 be weakly informative, with reliability matrix
\begin{equation*}
Q_m(a \mid j) = \begin{cases} 0.36, & a = j, \\ 0.32, & a \neq j, \end{cases} \qquad m=1, 2.
\end{equation*}
Let site 3 be highly reliable, with
\begin{equation*}
Q_3(a \mid j) = \begin{cases} 0.90, & a = j, \\ 0.05, & a \neq j. \end{cases}
\end{equation*}
Suppose the observed recommendations are $d = (1, 1, 2)$. Majority vote chooses $\xi_1$, because two sites recommend experiment 1. However, the Bayes rule chooses $\xi_2$.
The likelihood under oracle label $B=1$ is $\mathbb{P}(D = d \mid B = 1) = Q_1(1 \mid 1) Q_2(1 \mid 1) Q_3(2 \mid 1) = 036 \times 0.36 \times 0.05 = 0.00648$.
The likelihood under oracle label $B=2$ is $\mathbb{P}(D = d \mid B = 2) = Q_1(1 \mid 2) Q_2(1 \mid 2) Q_3(2 \mid 2) = 0.32 \times 0.32 \times 0.90 = 0.09216$.
The likelihood under oracle label $B=3$ is $\mathbb{P}(D = d \mid B = 3) = Q_1(1 \mid 3) Q_2(1 \mid 3) Q_3(2 \mid 3) = 0.32 \times 0.32 \times 0.05 = 0.00512$.
Since the priors are uniform, posterior probabilities are proportional to these likelihoods. The largest likelihood is obtained for $B = 2$. Hence $\delta_{\mathrm{MAP}}(1, 1, 2) = 2$. Thus Bayes fusion chooses $\xi_2$, whereas majority vote chooses $\xi_1$. Therefore, majority vote is strictly suboptimal.  This example shows that one reliable site can rationally outweigh several weak sites. Majority vote cannot express this because it assigns equal weight to all recommendations.

\vspace{-0.05cm}
\subsection{Multi-class Chernoff Bound}
For a finite set of $L$ candidate experiments, we also obtain a performance bound for the multi-class fusion rule.  First, consider the pairwise case.  For two distinct oracle labels $i$ and $j$, define the site-level pairwise Chernoff coefficient
\begin{equation}
c_{m,ij}(s) = \sum_{a=1}^L Q_m(a \mid j)^{1-s} Q_m(a \mid i)^s, \quad 0 \le s \le 1,
\label{eq:pairwise_Chernoff_coeff}
\end{equation}
where $c_{m,ij}(s)$ measures how difficult it is for site $m$'s message distribution to distinguish oracle label $j$ from oracle label $i$. If $Q_m(\cdot \mid j)$ and $Q_m(\cdot \mid i)$ are similar, then $c_{m,ij}(s)$ is close to 1 and it is difficult to distinguish labels $i$ and $j$.  If they are well separated, then $c_{m,ij}(s)$ is smaller.  Define the aggregate pairwise coefficient
\begin{equation*}
C_{ij}(s) = \prod_{m=1}^M c_{m,ij}(s).
\end{equation*}
Under conditional independence, this product structure is exact.  The corresponding pairwise Chernoff information is
\begin{align*}
\mathcal{C}_{ij} &= -\log \left[ \inf_{0 \le s \le 1} \prod_{m=1}^M c_{m,ij}(s) \right] \\
&= \sup_{0 \le s \le 1} \sum_{m=1}^M \left\{ -\log c_{m,ij}(s) \right\}.
\end{align*}
Thus, pairwise distinguishability accumulates additively across sites.  Next, we generalize this case to derive the multi-class Chernoff bound.  

Let $P_j(d) = \mathbb{P}(D = d \mid B = j) = \prod_{m=1}^M Q_m(d_m \mid j)$. The maximum a posteriori rule is
\begin{equation*}
\delta_{\mathrm{MAP}}(d) = \arg\max_{1 \le j \le L} \rho_j P_j(d).
\end{equation*}
We next derive the multi-class Chernoff error bound.
\begin{theorem}
For the maximum a posteriori fusion rule,
\begin{equation*}
\mathbb{P}\{ \delta_{\mathrm{MAP}}(D) \neq B \} \le \sum_{j=1}^L \sum_{\substack{i=1 \\ i \neq j}}^L \inf_{0 \le s \le 1} \rho_j^{1-s} \rho_i^s \prod_{m=1}^M c_{m,ij}(s),
\end{equation*}
where $c_{m,ij}(s)$ is given by \eqref{eq:pairwise_Chernoff_coeff}. 
\label{thrm:multiclass_Chernoff_bound}
\end{theorem}
\begin{IEEEproof} 
The probability of error is
\begin{equation*}
\mathbb{P}\{ \delta_{\mathrm{MAP}}(D) \neq B \} = \sum_{j=1}^L \rho_j \mathbb{P}_j \{ \delta_{\mathrm{MAP}}(D) \neq j \},
\end{equation*}
where $\mathbb{P}_j$ denotes probability conditional on $B = j$.  If $\delta_{\mathrm{MAP}}(d) \neq j$, then there exists some $i \neq j$ such that $\rho_i P_i(d) \ge \rho_j P_j(d)$. Therefore,
\begin{equation*}
\{ \delta_{\mathrm{MAP}}(D) \neq j \} \subseteq \bigcup_{\substack{i=1 \\ i \neq j}}^L \{ \rho_i P_i(D) \ge \rho_j P_j(D) \}.
\end{equation*}
By the union bound, we have
\begin{equation*}
\rho_j \mathbb{P}_j \{ \delta_{\mathrm{MAP}}(D) \neq j \} \le \sum_{\substack{i=1 \\ i \neq j}}^L \rho_j \mathbb{P}_j \{ \rho_i P_i(D) \ge \rho_j P_j(D) \}.
\end{equation*}
Fix $i \neq j$, and define $A_{ij} = \{ d : \rho_i P_i(d) \ge \rho_j P_j(d) \}$. Then $\rho_j \mathbb{P}_j(A_{ij}) = \sum_{d \in A_{ij}} \rho_j P_j(d)$. For $d \in A_{ij}$, $\rho_i P_i(d) \ge \rho_j P_j(d)$. Hence, for any $0 \le s \le 1$,
$$
\rho_j P_j(d) \le \left( \rho_j P_j(d) \right)^{1-s} \left( \rho_i P_i(d) \right)^{s}.
$$
Therefore, we get the following inequality:
\begin{align*}
\rho_j \mathbb{P}_j(A_{ij}) &\le \sum_d \left( \rho_j P_j(d) \right)^{1-s} \left( \rho_i P_i(d) \right)^{s}  \\
&= \rho_j^{1-s} \rho_i^s \sum_d P_j(d)^{1-s} P_i(d)^s.
\end{align*}
Using conditional independence, $P_j(d) = \prod_{m=1}^M Q_m(d_m \mid j)$ and $P_i(d) = \prod_{m=1}^M Q_m(d_m \mid i)$. Thus,
\begin{equation*}
P_j(d)^{1-s} P_i(d)^s = \prod_{m=1}^M Q_m(d_m \mid j)^{1-s} Q_m(d_m \mid i)^s.
\end{equation*}
Summing over all $d = (d_1, \ldots, d_M) \in \{1, \ldots, L\}^M$,
\begin{align*}
\sum_d P_j(d)^{1-s} P_i(d)^s &= \prod_{m=1}^M \sum_{a=1}^L Q_m(a \mid j)^{1-s} Q_m(a \mid i)^s  \\
&= \prod_{m=1}^M c_{m,ij}(s).
\end{align*}
This yields the following inequality:
\begin{equation*}
\rho_j \mathbb{P}_j(A_{ij}) \le \rho_j^{1-s} \rho_i^s \prod_{m=1}^M c_{m,ij}(s),
\end{equation*}
which holds for every $s \in [0, 1]$.  Hence, it also holds after taking the infimum over $s$. Summing over all $j$ and all $i \neq j$ yields the desired result.  
\end{IEEEproof} 

In the following, we derive the multi-class Chernoff bound for information-regret fusion.
\begin{theorem}
Under the finite-class fusion model, $R_\Delta(\delta_\Delta^*) = \mathbb{E} \left[ \Delta_{\delta_\Delta^*(D), B} \right]$ satisfies
\begin{equation*}
R_\Delta(\delta_\Delta^*) \le \sum_{j=1}^L \sum_{\substack{i=1 \\ i \neq j}}^L \Delta_{ij} \inf_{0 \le s \le 1} \rho_j^{1-s} \rho_i^s \prod_{m=1}^M c_{m,ij}(s).
\end{equation*}
\label{thrm:CB_info_regret_fusion}
\end{theorem}
\vspace{-0.5cm}
\begin{IEEEproof} 
The information-regret Bayes rule $\delta_\Delta^*$ minimizes $R_\Delta(\delta) = \mathbb{E} \left[ \Delta_{\delta(D), B} \right]$. Therefore, $R_\Delta(\delta_\Delta^*) \le R_\Delta(\delta_{\mathrm{MAP}})$. For the MAP rule,
\begin{equation*}
R_\Delta(\delta_{\mathrm{MAP}}) = \sum_{j=1}^L \rho_j \sum_{\substack{i=1 \\ i \neq j}}^L \Delta_{ij} \mathbb{P}_j \{ \delta_{\mathrm{MAP}}(D) = i \}.
\end{equation*}
The event $\{ \delta_{\mathrm{MAP}}(D) = i \}$ is contained in $\{ \rho_i P_i(D) \ge \rho_j P_j(D) \}$.  Using the same Chernoff argument, we get
\begin{equation*}
\rho_j \mathbb{P}_j \{ \delta_{\mathrm{MAP}}(D) = i \} \le \inf_{0 \le s \le 1} \rho_j^{1-s} \rho_i^s \prod_{m=1}^M c_{m,ij}(s).
\end{equation*}
Consequently,
\begin{equation*}
R_\Delta(\delta_\Delta^*) \le \sum_{j=1}^L \sum_{\substack{i=1 \\ i \neq j}}^L \Delta_{ij} \inf_{0 \le s \le 1} \rho_j^{1-s} \rho_i^s \prod_{m=1}^M c_{m,ij}(s).
\vspace{-0.5cm}
\end{equation*}
\end{IEEEproof} 
\thrmref{thrm:CB_info_regret_fusion} identifies the difficult pairs of candidate experiments, i.e., it shows which pairs of experiments are most likely to cause fusion errors with substantial information loss.  Essentially, it shows that the regret bound is dominated by pairs of experiments for which selecting one instead of the other causes large information loss and the local recommendation channels provide weak separation.  Pairs $(i, j)$ with large regret $\Delta_{ij}$ and large Chernoff coefficient $c_{m,ij}(s)$ dominate the upper bound.  Hence, the distributed design problem is governed not merely by the number of local votes, but by the pairwise distinguishability of candidate experiments under the local recommendation channels.

A consequence of \thrmref{thrm:CB_info_regret_fusion} is that it governs exponential decay under pairwise separation.  Suppose that for every pair $i \neq j$, there exists $s_{ij} \in [0, 1]$ and constants $\kappa_{m,ij} \ge 0$ such that $c_{m,ij}(s_{ij}) \le e^{-\kappa_{m,ij}}$.  Then, 
\begin{equation*}
\prod_{m=1}^M c_{m,ij}(s_{ij}) \le \exp \left\{ -\sum_{m=1}^M \kappa_{m,ij} \right\}.
\end{equation*}
Therefore, we get
\begin{equation*}
R_\Delta(\delta_\Delta^*) \le \sum_{j=1}^L \sum_{\substack{i=1 \\ i \neq j}}^L \Delta_{ij} \rho_j^{1-s_{ij}} \rho_i^{s_{ij}} \exp \left\{ -\sum_{m=1}^M \kappa_{m,ij} \right\}.
\end{equation*}
In particular, if $\sum_{m=1}^M \kappa_{m,ij} \to \infty$ for every $i \neq j$, then $R_\Delta(\delta_\Delta^*) \to 0$, provided the regret values $\Delta_{ij}$ remain bounded.  Thus, the Bayes-fused design approaches the centralized oracle design when the local recommendation collectively separate every pair of oracle labels.

This can also be interpreted as an additive design-separation criterion.  The site-level pairwise information contribution is $-\log c_{m,ij}(s)$.  The aggregate pairwise separation between oracle labels $i$ and $j$ is $\sum_{m=1}^M \left\{ -\log c_{m,ij}(s) \right\}$.  Thus, the $L$-candidate design problem naturally leads to the criterion: \emph{separate all consequential pairs $(i, j)$.}  A pair is consequential when $\Delta_{ij}$ is large. A pair is statistically difficult when $c_{m,ij}(s)$ is close to $1$ for many sites. Therefore the most important pairs are those for which $\Delta_{ij} \prod_{m=1}^M c_{m,ij}(s)$ is large.  This gives a principled distributed design objective:
\begin{equation*}
\max \min_{\substack{i, j \\ i \neq j}} \left[ \sum_{m=1}^M -\log c_{m,ij}(s) \right],
\end{equation*}
or, in regret-weighted form,
\begin{equation*}
\min \sum_{j=1}^L \sum_{\substack{i=1 \\ i \neq j}}^L \Delta_{ij} \rho_j^{1-s} \rho_i^s \prod_{m=1}^M c_{m,ij}(s).
\end{equation*}
This shows that distributed experimental design with $L$ candidates is not a voting problem. It is a regret-weighted multi-class information-fusion problem.

\begin{remark}
MAP fusion can perform competitively, and in some settings may outperform the regret-weighted rule.  This is not inconsistent with the proposed framework.  MAP fusion targets oracle-label recovery, whereas regret-weighted fusion targets posterior expected information loss.  When the information-regret matrix is close to ordinary $0-1$ loss, or when the posterior distribution is concentrated on a single oracle label, the two objectives become nearly aligned.  In such cases, the simpler MAP rule can be as effective as, or slightly better than, the regret-weighted rule in finite samples, especially because the latter requires estimation of an additional regret matrix.  The advantage of the regret-weighted formulation is that it provides the correct decision criterion when different wrong experimental choices have meaningfully different information losses, and therefore, it uniformly dominates MAP in such a setting.  Numerical results in \secref{sec:numerical_results} substantiate this phenomenon by identifying two distinct operational regimes. 
\label{remark:MAP_regret-weighted}
\end{remark}

\begin{remark}
The majority vote rule considered here is related to classical voting procedures, but it should not be confused with a full Condorcet's method \cite{Wallis2014}, \cite[Chapter 2]{Gehrlein2011}.  In our formulation, each local site reports only a single recommended experiment, so the resulting rule is essentially a plurality-type aggregation of first-choice recommendations rather than a pairwise majority comparison based on complete rankings.  This distinction is important because the proposed Bayes fusion rule is not intended to satisfy voting-theoretic criteria; instead, it uses site-specific recommendation probabilities and information-regret weights to select the experiment with smallest posterior expected loss. 
\label{remark:Condorcet}
\end{remark}

\begin{remark}
The finite-class fusion model developed in this paper can also be viewed through the lens of Blackwell's comparison of experiments \cite{Blackwell1951,Blackwell1953}.  The vector of local recommendations $D = (D_1, \ldots, D_M)$ defines an experiment about the oracle label $B$, with transition probabilities determined by the recommendation channels $Q_m(a \mid j)$.  If one communication scheme is a garbling of another \cite[Theorem 1]{deOliveira2018}, then it is less informative in the Blackwell sense and cannot improve Bayes risk uniformly over decision problems.  Our analysis does not seek a complete Blackwell ordering of distributed recommendation schemes.  Instead, it fixes the induced recommendation experiment and studies the Bayes rule associated with the information-regret loss matrix $\Delta$.
\label{remark:Blackwell}
\end{remark}

\section{Illustrative examples}\label{sec:illustrative_examples}
In this section, we introduce two representative experimental design models to illustrate the proposed regret-weighted Bayes fusion framework.  The Gaussian linear model provides a precision-based setting in which local and centralized information gains can be compared directly, while the binary-response model demonstrates the applicability of the method to nonlinear information-gain calculations involving discrete outcomes.  These examples demonstrate how local design recommendations induce site-specific recommendation, which can be combined with oracle-label probabilities and information-regret weights to obtain a global Bayes action. Thus, the fusion step is not an unweighted aggregation of local choices, but a posterior risk minimization procedure calibrated to the centralized expected information-gain criterion.

\subsection{Gaussian linear experimental design}\label{subsec:Gaussian_model}
In this setting, each candidate experiment contributes a site-specific precision increment for a scalar Gaussian parameter. Each site recommends the locally most informative experiment, whereas the centralized oracle selects the experiment with the largest aggregate EIG across sites.

Consider the scalar Gaussian model $\theta \sim \mathcal{N}(0, \tau^2)$. At site $m$, candidate experiment $\xi_\ell$ produces $Y_{m\ell} = a_{m\ell}\theta + \varepsilon_{m\ell}$, $\varepsilon_{m\ell} \sim \mathcal{N}(0, \sigma_{m\ell}^2)$, independently across sites conditional on $\theta$. Define the precision contribution $r_{m\ell} = a_{m\ell}^2/\sigma_{m\ell}^2$.  If experiment $\xi_\ell$ is implemented centrally across all sites, the total precision contribution is $R_\ell = \sum_{m=1}^M r_{m\ell}$. The centralized EIG is
\begin{equation}
U(\xi_\ell) = \frac{1}{2} \log \left( 1 + \tau^2 R_\ell \right) = \frac{1}{2} \log \left( 1 + \tau^2 \sum_{m=1}^M r_{m\ell} \right).
\label{eq:Gauss_EIG}
\end{equation}
The centralized oracle label is 
\begin{equation*}
B = \arg\max_{1 \le \ell \le L} \sum_{m=1}^M r_{m\ell}.
\end{equation*}
From \eqref{eq:Gauss_EIG}, it is clear that maximizing $U(\xi_\ell)$ is equivalent to maximizing $\sum_{m=1}^M r_{m\ell}$.  Site $m$'s local EIG is $U_m(\xi_\ell) = \frac{1}{2} \log(1 + \tau^2 r_{m\ell})$.  Thus, its local recommendation is $D_m = \arg\max_{1 \le \ell \le L} r_{m\ell}$.

This example shows the fundamental distinction between local and centralized design. Site $m$ chooses the design with largest local precision $r_{m\ell}$, whereas the centralized oracle chooses the design with largest aggregate precision $\sum_{m=1}^M r_{m\ell}$. Local agreement is not required for centralized optimality. A design may be globally optimal because it performs consistently well across many sites, even if it is not the top local choice at most individual sites.  The regret matrix is given by 
\begin{equation*}
\begin{split}
\Delta_{ij}
&= \mathbb{E}\Biggl[\frac{1}{2}\log\!\left(1+\tau^2\sum_{m=1}^M r_{mj}\right)
\\
&\qquad\qquad
-
\frac{1}{2}\log\!\left(1+\tau^2\sum_{m=1}^M r_{mi}\right)\Bigg| B = j\Biggr].
\end{split}
\end{equation*}
The Bayes information-regret fusion rule is, therefore,
\begin{equation*}
\delta_\Delta^*(d) = \arg\min_{1 \le i \le L} \sum_{j=1}^L \Delta_{ij} \rho_j \prod_{m=1}^M Q_m(d_m \mid j),
\end{equation*}
which uses the local recommendations $d_m$, the reliability matrices $Q_m$, and the centralized information-regret gaps $\Delta_{ij}$. It is generally different from choosing the experiment with the most local recommendations.

\subsection{Binary-response experimental design}\label{subsec:binary_model}
Now suppose site $m$, under experiment $\xi_\ell$, observes a binary response $Y_{m\ell} \in \{0, 1\}$, with $\mathbb{P}(Y_{m\ell} = 1 \mid \theta, \xi_\ell) = p_{m\ell}(\theta)$.
In a logistic model,
\begin{equation*}
p_{m\ell}(\theta) = \frac{\exp \eta_{m\ell}(\theta)}{1 + \exp \eta_{m\ell}(\theta)}.
\end{equation*}
Let
\begin{equation*}
\bar{p}_{m\ell} = \mathbb{E}_\theta [p_{m\ell}(\theta)] = \int p_{m\ell}(\theta) p(\theta) \, d\theta.
\end{equation*}
The prior predictive distribution of $Y_{m\ell}$ is Bernoulli with success probability $\bar{p}_{m\ell}$.  The local EIG is $U_m(\xi_\ell) = I(\theta; Y_{m\ell} \mid \xi_\ell)$. For a binary observation, $I(\theta; Y_{m\ell} \mid \xi_\ell) = h(\bar{p}_{m\ell}) - \mathbb{E}_\theta [h(p_{m\ell}(\theta))]$, where $h(x) = -x \log x - (1-x) \log (1-x)$ is the binary entropy function.  Thus, site $m$'s recommendation is given by 
\begin{equation*}
D_m = \arg\max_{1 \le \ell \le L} \left[ h(\bar{p}_{m\ell}) - \mathbb{E}_\theta \{ h(p_{m\ell}(\theta)) \} \right].
\end{equation*}

For the centralized experiment $\xi_\ell$, the full response vector is $Y_\ell = (Y_{1\ell}, \ldots, Y_{M\ell})$. Conditional on $\theta$, the components are independent. The centralized EIG is
\begin{align*}
U(\xi_\ell) &= I(\theta; Y_\ell \mid \xi_\ell) = H(Y_\ell \mid \xi_\ell) - \mathbb{E}_\theta \left[ \sum_{m=1}^M h(p_{m\ell}(\theta)) \right] \\
&= -\!\!\!\!\!\! \sum_{y \in \{0, 1\}^M} \!\!\!\!\!\!   p_\ell(y) \log p_\ell(y) - \!\!\! \int \left[ \sum_{m=1}^M h(p_{m\ell}(\theta)) \right] p(\theta) \, d\theta.
\end{align*}
The Bayes information-regret fusion rule is again
\begin{equation*}
\delta_\Delta^*(d) = \arg\min_{1 \le i \le L} \sum_{j=1}^L \Delta_{ij} \rho_j \prod_{m=1}^M Q_m(d_m \mid j).
\end{equation*}
The binary-response example shows that the same finite-$L$ fusion theory applies beyond Gaussian linear models. The local utilities and centralized utilities may be nonlinear, but the fusion rule depends only on the induced oracle prior $\rho_j$, the local recommendation matrices $Q_m(a \mid j)$, and the information-regret matrix $\Delta_{ij}$.

\section{Numerical results}\label{sec:numerical_results}
We now examine the comparative performance of three fusion rules\textemdash majority vote, MAP fusion, and regret-Bayes fusion\textemdash under two experimental setups: baseline and calibrated loss-sensitive.  These three fusion rules can agree in highly symmetric settings, but they generally differ when the sites are heterogeneous, the local recommendation mechanisms have unequal reliability, or the information loss from different wrong choices is not uniform.  The numerical studies are designed to make these distinctions visible.  In each experimental setting, the performance of each rule is evaluated using oracle-selection accuracy and average information regret.

For each setting, performance is evaluated by Monte Carlo simulation using independent training and testing protocols. In the Gaussian design experiment, for each fixed number of sites, $60{,}000$ training replicates and $60{,}000$ independent testing replicates were generated. In the binary-response experiment, for each fixed number of sites, $2{,}500$ training replicates and $2{,}500$ independent testing replicates were generated. Each training replicate produced a distributed design instance, the local site recommendations, and the corresponding centralized oracle experiment computed from the full-information EIG. The training replicates were used to estimate the oracle-label probabilities, the site-specific recommendation mechanisms, and the information-regret matrix. These estimated quantities were then held fixed and applied to the independent test replicates. On the test sample, majority vote, MAP fusion, and regret-weighted Bayes fusion were compared using oracle-selection accuracy and average information regret. In the Gaussian model, the local and centralized utilities were computed directly from the precision-based information-gain formula. In the binary-response model, the required expectations over the Gaussian prior were evaluated using Gauss--Hermite quadrature with 35 quadrature nodes; the quadrature nodes and weights were transformed to match the $\mathcal{N}(0, \tau^2)$ prior. The centralized binary-response information gain was computed by combining this quadrature with exact enumeration of all binary response patterns across the distributed sites.

For the Gaussian model, each simulation replicate represents one distributed design instance. In each replicate, the local precision contributions vary across sites and candidate experiments, generating both site heterogeneity and uncertainty about the centralized oracle design. The local recommendation from each site is obtained by selecting the experiment with the largest local precision contribution. The centralized utility of each candidate is computed from the total precision accumulated across sites, and the oracle design is the candidate with the largest centralized utility. A separate training sample is used to estimate the prior probability of each oracle label, the site-specific recommendation behavior, and the information-regret matrix. These estimated quantities are held fixed and used to evaluate the three fusion rules on an independent test sample. This train-test separation avoids evaluating the fusion rules on the same simulated instances used to estimate the recommendation channels.

For the binary-response model, the unknown parameter is assigned a Gaussian prior, and each candidate experiment corresponds to a binary-response design point.  At each site, the probability of success depends on the unknown parameter, the candidate design point, and a site-specific offset.  Thus, different sites may prefer different candidate experiments because their response curves are shifted relative to one another. Each site computes the local EIG for each candidate design point and recommends the locally best experiment. The centralized oracle computes the EIG of the full vector of binary responses across all sites and selects the globally best candidate.

The binary-response model is computationally more intensive than the Gaussian model because the centralized EIG depends on the joint predictive distribution of the binary response vector. For a fixed number of sites, this distribution is evaluated by enumerating all binary response patterns and integrating over the prior distribution of the unknown parameter using numerical quadrature.  As such, the binary-response study is run for a moderate range of site counts. This is not a limitation of the fusion theory; rather, it reflects the cost of exactly evaluating the centralized information gain in the model. The same train-test protocol is used as in the Gaussian experiment.  

\subsection{Baseline setup}\label{subsec:baseline}
In this setting, oracle-label accuracy and average information regret are aligned. 
\begin{figure*}[t]
    \centering
    \scriptsize
    \subfloat[Average regret vs. number of sites.]{%
        \includegraphics[width=0.48\textwidth]{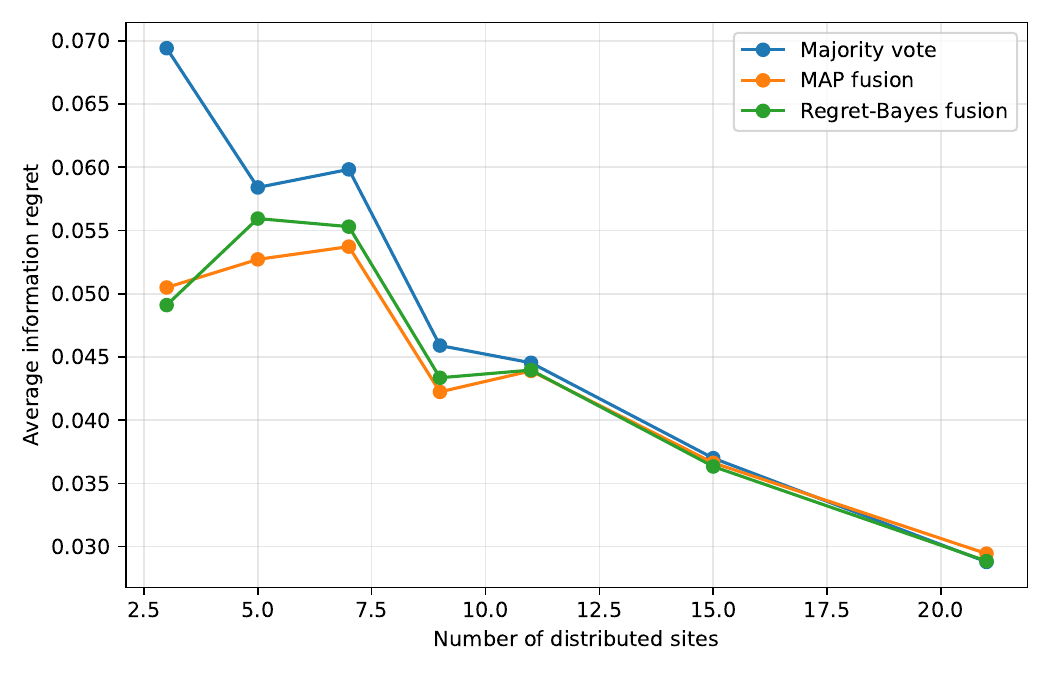}
        \label{fig:Gauss_average_regret_vs_sites}
    }
    \hfill
    \subfloat[Oracle-selection accuracy vs. number of sites.]{%
        \includegraphics[width=0.48\textwidth]{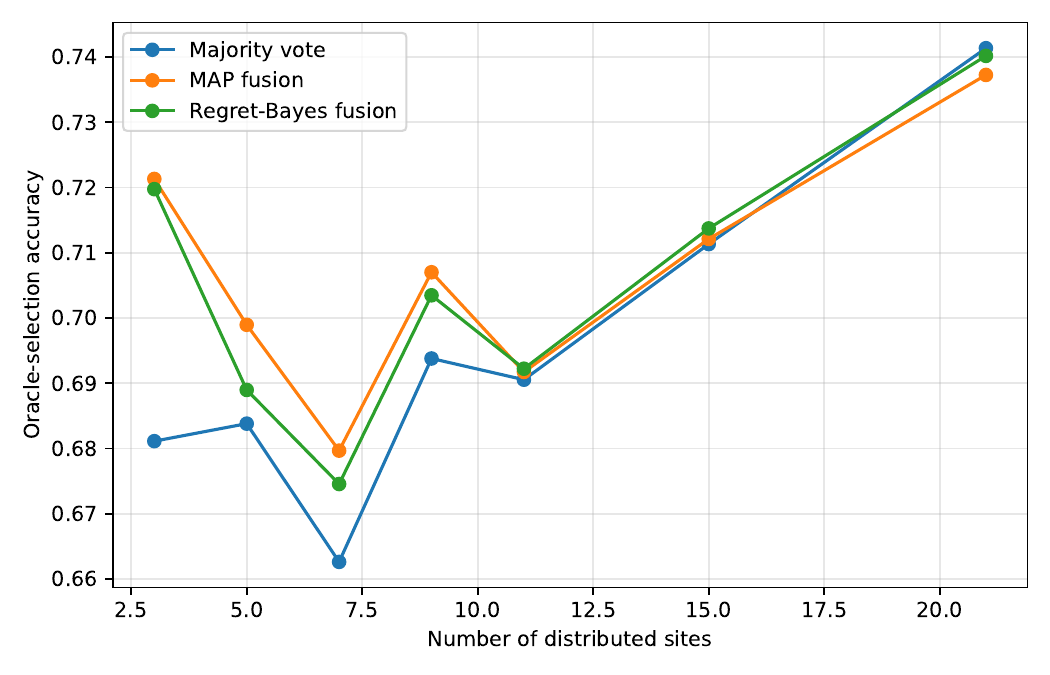}
        \label{fig:Gauss_oracle_accuracy_vs_sites}
    }
    \caption{Performance of the Gaussian model in the baseline setup.}
    \label{fig:one}
\end{figure*}
\figref{fig:Gauss_average_regret_vs_sites} shows the average information regret versus the number of distributed sites for the Gaussian model.  A smaller value indicates that the selected experiment is closer to the centralized oracle in terms of EIG.  Majority vote provides the baseline behavior obtained by na\"{i}ve aggregation.  The MAP curve shows the effect of reliability-weighted probabilistic fusion based on the estimated local recommendation channels, while the regret-Bayes curve shows the effect of additionally incorporating the information-regret matrix.  When regret-Bayes fusion attains the smallest average regret, the result indicates that, in this setting, directly minimizing posterior expected information loss is better aligned with the experimental-design objective than either counting local votes or maximizing the posterior probability of the oracle label.

\figref{fig:Gauss_oracle_accuracy_vs_sites} shows the oracle-selection accuracy versus the number of distributed sites for the Gaussian model.  MAP fusion is expected to perform strongly on this metric because it is specifically designed to maximize posterior oracle-label probability. Regret-Bayes fusion may sometimes have slightly lower oracle-selection accuracy than MAP fusion, but this should not be interpreted as a weakness if it simultaneously achieves lower average information regret. Such a pattern would show that the regret-based rule may sacrifice exact oracle-label recovery in some cases in order to avoid decisions with larger information loss.  Comparing \figref{fig:Gauss_average_regret_vs_sites} and \figref{fig:Gauss_oracle_accuracy_vs_sites} shows that accuracy and regret measure related but distinct objectives, and the proposed method is designed for the latter.

The Gaussian model shows that, when local sites are heterogeneous, reliability-weighted and regret-weighted fusion can improve the fused design by increasing oracle-selection accuracy and reducing average information regret relative to na\"{i}ve voting.  Majority vote treats a recommendation from every site as equally informative. This is appropriate only under strong symmetry assumptions. When those assumptions fail, majority vote can be pulled toward locally popular designs that are not globally best, or toward designs whose mistakes carry larger information cost.  MAP fusion corrects part of this problem by learning how each site's recommendation relates to the centralized oracle. Regret-Bayes fusion corrects the remaining part by incorporating the severity of selecting each candidate when another candidate is oracle-optimal.

\begin{figure*}[t]
    \centering
    \scriptsize
    \subfloat[Average regret vs. number of sites.]{%
        \includegraphics[width=0.48\textwidth]{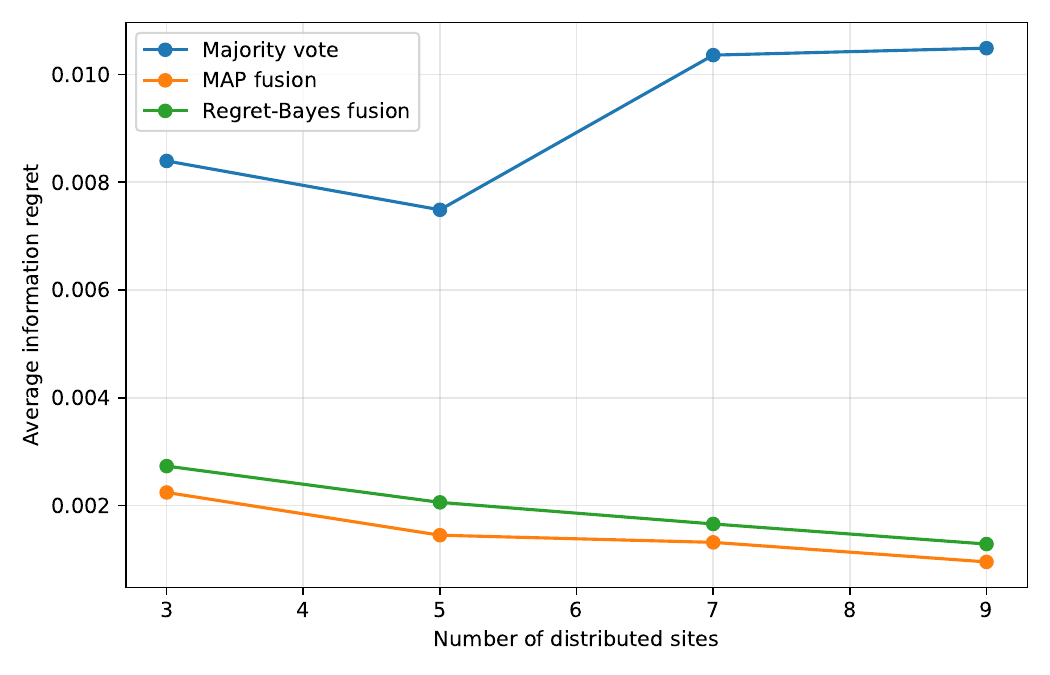}
        \label{fig:binary_response_regret_vs_sites}
    }
    \hfill
    \subfloat[Oracle-selection accuracy vs. number of sites.]{%
        \includegraphics[width=0.48\textwidth]{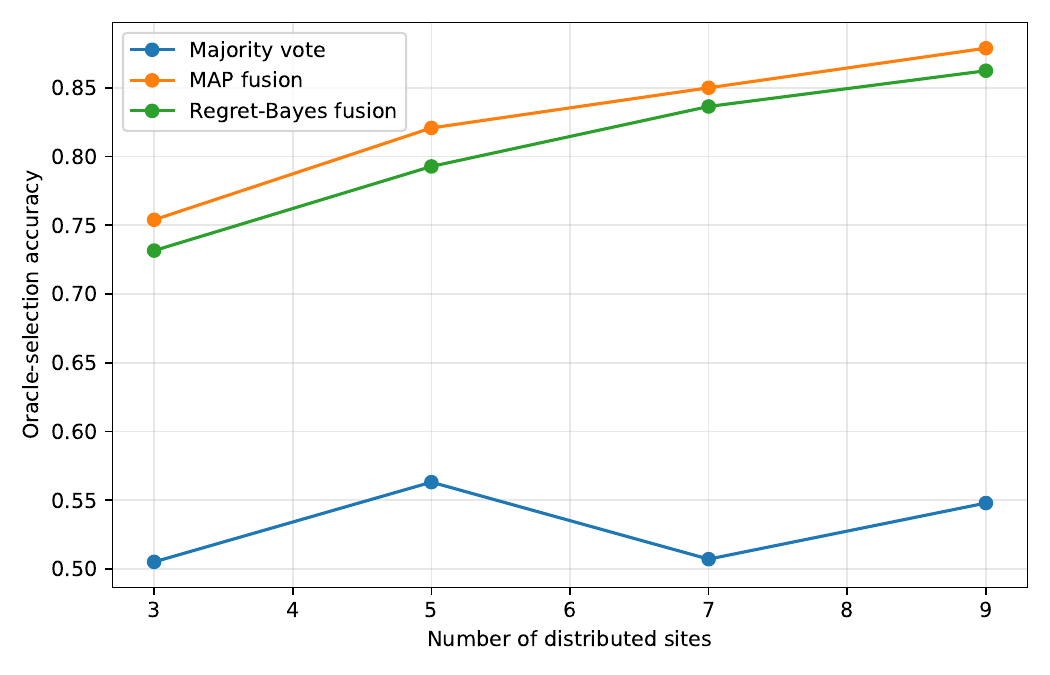}
        \label{fig:binary_response_accuracy_vs_sites}
    }
    \caption{Performance of the binary-response model in the baseline setup.}
    \label{fig:one}
\end{figure*}
\figref{fig:binary_response_regret_vs_sites} shows the average information regret versus the number of distributed sites for the binary-response model.  The expected qualitative pattern is that regret-Bayes fusion should be most competitive on average regret, because it is the only method among the three that explicitly uses the regret matrix. MAP fusion may improve on majority vote by accounting for site reliability, but it still treats all oracle-label errors according to a classification objective. Majority vote is the least adaptive because it ignores both site reliability and the unequal consequences of different wrong design choices.  Therefore, when the regret curve for regret-Bayes fusion lies below the other curves, this is direct evidence that the proposed decision rule is better aligned with the goal of preserving centralized information gain.

\figref{fig:binary_response_accuracy_vs_sites} shows the oracle-selection accuracy versus the number of distributed sites for the binary-response model.  If MAP fusion has the highest oracle-selection accuracy but regret-Bayes fusion has the lowest average regret, the results support the central claim of the paper: exact recovery of the oracle label and minimization of information loss are not the same. In experimental design, a rule that occasionally selects a non-oracle design with nearly identical information gain may be preferable to a rule that achieves slightly higher oracle-label accuracy but sometimes makes more costly mistakes. The binary-response model is particularly useful for demonstrating this distinction because the information landscape across candidate design points can be relatively flat in some regions and sharply separated in others.

The binary-response model shows that the proposed method is meaningful even when the information structure is nonlinear. In binary-response design, the utility of a candidate experiment is shaped by the prior, the response probability curve, the candidate design point, and the site-specific offset. Local recommendations can therefore disagree for scientifically meaningful reasons. A frequency-based rule cannot distinguish between disagreement caused by noise and disagreement caused by heterogeneous information geometry. The Bayes fusion rules address this by estimating the relationship between local recommendations and centralized oracle choices. The regret-weighted rule further accounts for how costly each possible disagreement is in terms of information gain.

\subsection{Calibrated loss-sensitive setup}\label{subsec:loss-sensitive}
The preceding numerical study showed that when the information-regret structure is closely aligned with oracle-label accuracy, MAP fusion can be superior.   We now conduct a calibrated experiment to examine a more loss-sensitive regime, wherein some oracle labels are less frequent but much more costly to miss.  Here, we consider $L = 4$ candidate experiments and vary the number of distributed sites $M$. For each $M$, a training sample was used to estimate $\rho_j$, $Q_m(a \mid j)$, and $\Delta_{ij}$, and an independent test sample was used to evaluate majority vote, MAP fusion, and regret-weighted Bayes fusion.  The model was calibrated so that some candidate experiments were difficult to distinguish from local recommendations, while certain wrong selections incurred much larger information regret. This calibration creates a setting in which MAP can remain effective for oracle-label recovery, but need not minimize average information loss.  In both the Gaussian and binary-response models, the local recommendation channels were chosen so that two candidate experiments are difficult to distinguish from local recommendations alone, while the information-regret matrix assigns substantially larger loss to one direction of confusion. This creates a stress-test in which MAP continues to target the most probable oracle label, whereas regret-weighted Bayes fusion explicitly accounts for the asymmetric information loss associated with different wrong experimental choices.

\begin{figure*}[t]
    \centering
    \scriptsize
    \subfloat[Average regret vs. number of sites.]{%
        \includegraphics[width=0.48\textwidth]{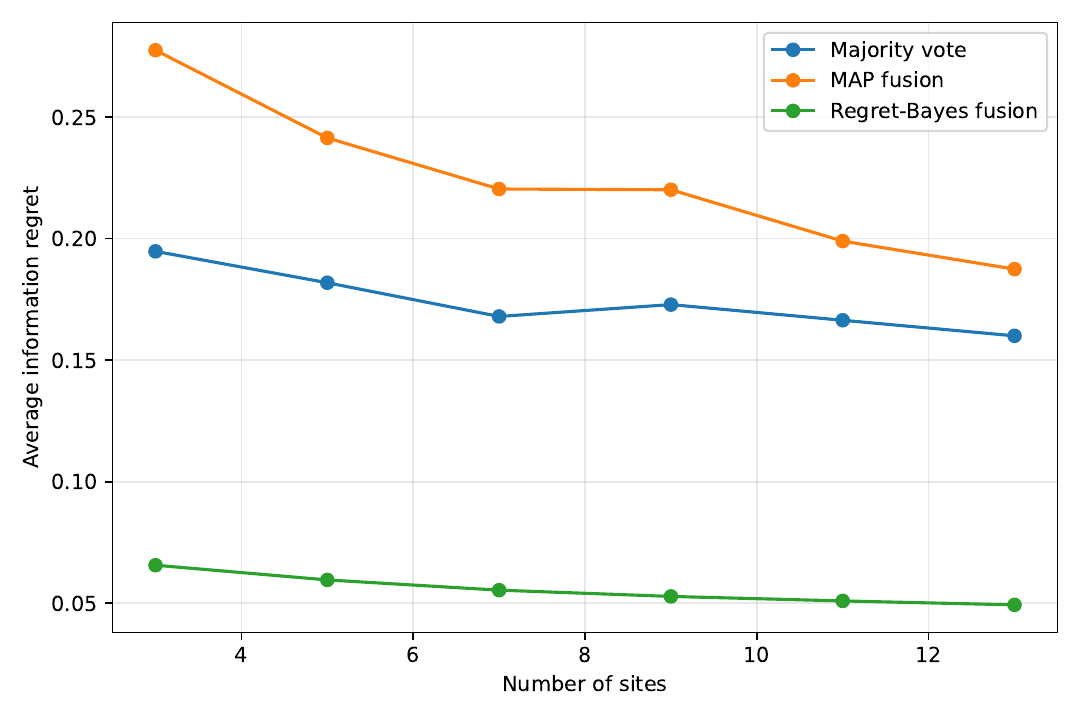}
        \label{fig:Gaussian_average_regret_stress}
    }
    \hfill
    \subfloat[Oracle-selection accuracy vs. number of sites.]{%
        \includegraphics[width=0.48\textwidth]{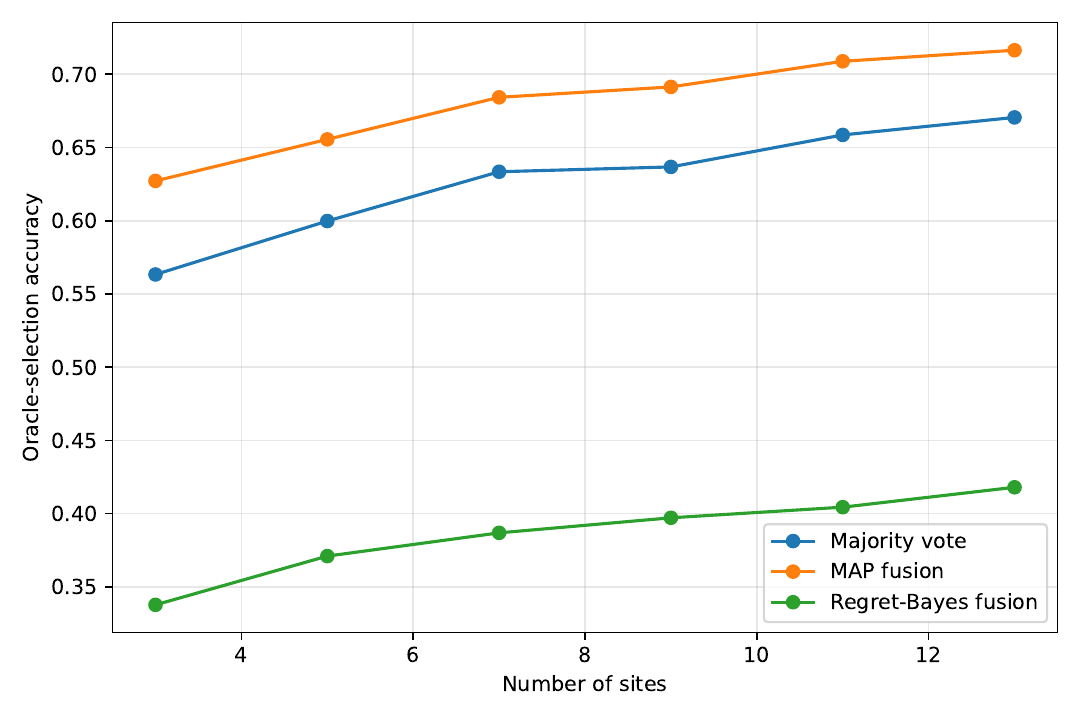}
        \label{fig:Gaussian_oracle_accuracy_stress}
    }
    \caption{\centering Performance of the Gaussian model in the calibrated loss-sensitive setup.}
    \label{fig:one}
\end{figure*}
For the Gaussian model, the average-regret plot in \figref{fig:Gaussian_average_regret_stress} shows a clear separation between the three fusion rules. MAP fusion, although probabilistically well-calibrated for oracle-label recovery, has the largest average information regret across all values of $M$. Majority vote improves over MAP in regret because it is less strongly biased toward the most probable oracle label, but it still ignores the estimated regret structure. Regret-weighted Bayes fusion achieves the smallest average regret uniformly over the displayed range of sites.  It can be inferred that, when wrong choices have unequal information consequences, minimizing posterior expected regret can preserve centralized information gain more effectively than maximizing the posterior probability of the oracle label.

The Gaussian oracle-selection accuracy plot in \figref{fig:Gaussian_oracle_accuracy_stress} shows the complementary behavior. MAP fusion has the highest oracle-selection accuracy for every value of $M$, followed by majority vote, while regret-weighted Bayes fusion has lower oracle-label accuracy. This is not a contradiction. In this experiment, the regret-weighted rule deliberately sacrifices some probability of selecting the exact oracle label in order to avoid high-regret errors. Thus, the Gaussian results separate the two performance criteria: MAP is preferable for label recovery, whereas regret-weighted Bayes fusion is preferable for minimizing information loss.

\begin{figure*}[htbp!]
    \centering
    \scriptsize
    \subfloat[Average regret vs. number of sites.]{%
        \includegraphics[width=0.48\textwidth]{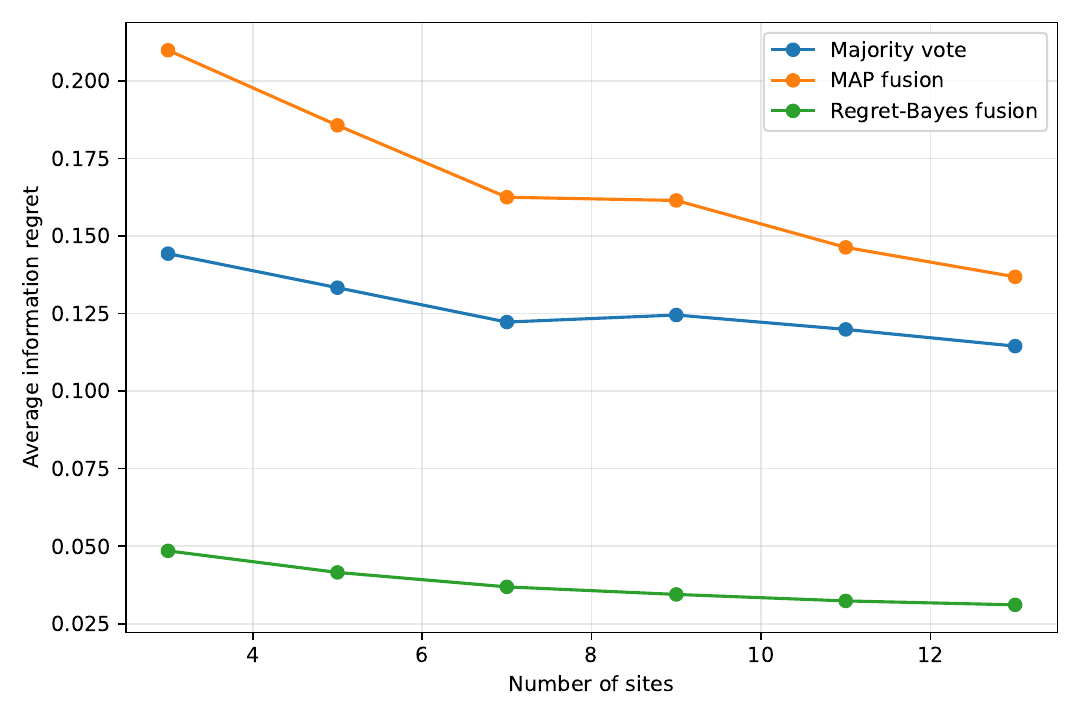}
        \label{fig:binary_average_regret_stress}
    }
    \hfill
    \subfloat[Oracle-selection accuracy vs. number of sites.]{%
        \includegraphics[width=0.48\textwidth]{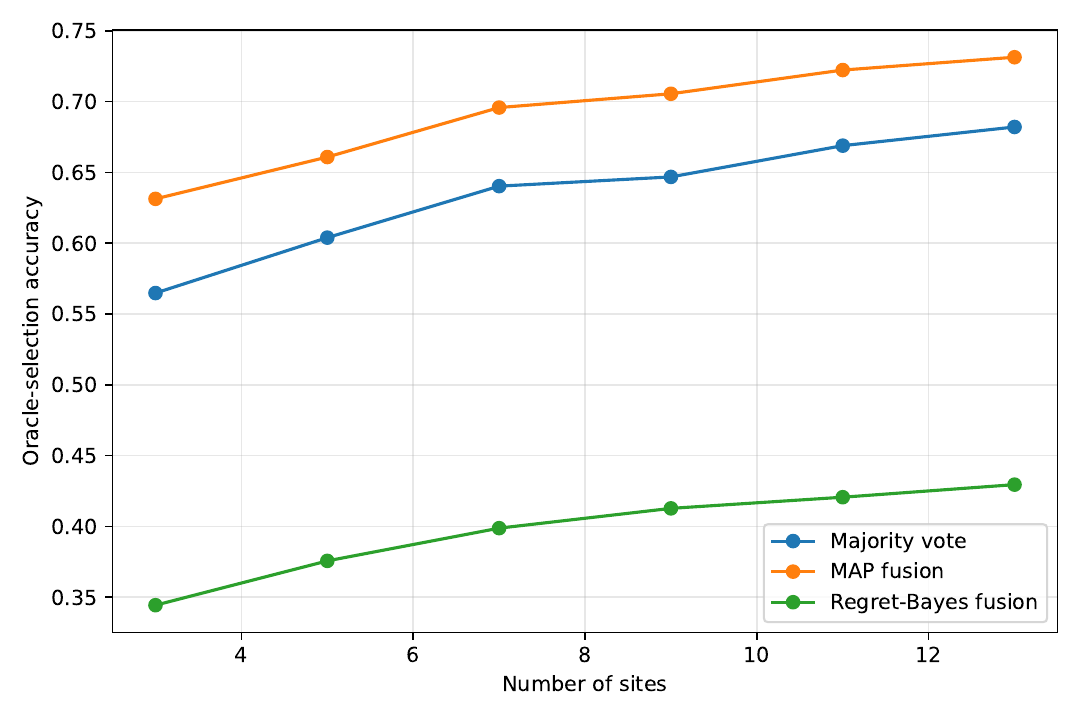}
        \label{fig:binary_oracle_accuracy_stress}
    }
    \caption{\centering Performance of the binary-response model in the calibrated loss-sensitive setup.}
    \label{fig:one}
\end{figure*}
The binary-response average-regret plot in \figref{fig:binary_average_regret_stress} exhibits the same qualitative conclusion in a nonlinear information-gain setting.  MAP fusion again has the largest average information regret, majority vote is intermediate, and regret-weighted Bayes fusion gives the lowest regret for all numbers of sites. The persistence of this ordering in the binary-response model is important because the utilities are no longer simple precision-based quantities; they arise from EIG under a discrete response model. Hence, the advantage of regret-weighted fusion is not an artifact of the Gaussian construction, but reflects the role of the regret matrix in the fusion decision.  The binary-response oracle selection accuracy in \figref{fig:binary_oracle_accuracy_stress} confirms that the oracle-label accuracy of MAP increases with the number of sites and remains higher than both majority vote and regret-weighted Bayes fusion. However, this higher label accuracy does not translate into lower average regret because the costly confusions are not treated differently by MAP.  

The numerical results substantiate the claim that multi-candidate distributed experimental design should not be treated as a voting problem, which discards site reliability, pairwise confusability, and the information loss caused by different mistakes. The regret-Bayes rule uses these quantities through the oracle-label probabilities, local recommendation channels, and information-regret matrix.  Secondly, oracle-selection accuracy alone is not sufficient for evaluating a distributed design rule.  Accuracy is useful, but it is a classification metric.  In experimental design, the more important question is how much information is lost by the selected experiment relative to the centralized oracle, so posterior expected information regret is the relevant criterion.  A method with the highest accuracy need not have the smallest regret, and a method with slightly lower accuracy may be preferable if its errors are less costly.  This distinction is precisely why the multi-candidate setting is richer than the binary case and why a regret-weighted formulation is needed.  The Gaussian and binary-response examples show that the full information-regret structure of a multi-class decision problem changes the behavior of the fusion rule, the interpretation of numerical performance, and the practical recommendation for distributed experimental design.

\section{Conclusion}\label{sec:conclusion}
We developed a regret-weighted Bayes fusion framework for distributed experimental design with multiple candidate experiments.  The proposed rule combines local site recommendations by accounting for site-specific reliability, prior plausibility of the centralized oracle design, and the information loss incurred by selecting a suboptimal experiment.  This extends the earlier two-experiment distributed experimental design formulation \cite{Nagananda2026c} from a binary threshold rule to a multi-class Bayes decision problem with a full information-regret structure. The theoretical results show that majority vote is optimal only under restrictive symmetry assumptions, while the numerical studies show that the preferred fusion rule depends on the relation between oracle-label accuracy and information regret.  When these criteria are aligned, MAP fusion is effective because recovering the most probable oracle label also preserves information gain.  When the regret matrix is asymmetric or highly uneven, however, this alignment can fail: a less probable oracle label may correspond to a much larger loss if missed.  In such loss-sensitive settings, regret-weighted Bayes fusion is preferable because it uses the full posterior risk rather than the posterior mode alone.

Future work may consider sequential distributed design, where sites update their recommendations over multiple rounds as new data are collected. Another useful direction is to allow richer messages from local sites, such as ranked candidate lists, local utility scores, or uncertainty summaries, instead of a single recommended experiment. Robust fusion under misspecified recommendation channels, dependent sites, privacy constraints, and communication limits also deserves further study.


\balance
\bibliographystyle{IEEEtran}
\bibliography{/Users/kgnagananda/Documents/Work/collaborations/pdx/research/references/research_pdx.bib}

\vfill

\end{document}